\documentclass[11pt]{amsart}
\usepackage{amsthm}
\usepackage{amsmath}
\usepackage{amssymb}
\usepackage{comment}
\usepackage{tikz}
\usepackage{enumitem}
\usepackage{amsfonts}
\usepackage{mathtools}
\usepackage{hyperref}
\usepackage{todonotes}
\usepackage[utf8]{inputenc}
\usepackage[left=1.5in, right=1.5in]{geometry}

\DeclareMathOperator{\conv}{Conv}
\DeclareMathOperator{\weak}{Weak}
\DeclareMathOperator{\Bot}{bot} 
\DeclareMathOperator{\Top}{top}
\DeclareMathOperator{\fl}{fl}
\DeclareMathOperator{\asc}{asc}
\DeclareMathOperator{\argmax}{argmax}
\DeclareMathOperator{\Perm}{Perm}

\newcommand{\ph}{\varphi}

\newcommand{\bs}{\boldsymbol}
\renewcommand{\tilde}{\widetilde}

\newcommand{\R}{\mathbb{R}}

\newcommand\precdot{\mathrel{\ooalign{$\prec$\cr
  \hidewidth\raise0.001ex\hbox{$\cdot\mkern0.6mu$}\cr}}}

\newtheorem{theorem}{Theorem}[section]
\newtheorem{def-prop}[theorem]{Definition-Proposition}
\newtheorem{prop}[theorem]{Proposition}

\newtheorem{lemma}[theorem]{Lemma}
\newtheorem{cor}[theorem]{Corollary}
\newtheorem{introthm}{Theorem}

\theoremstyle{definition}
\newtheorem{ex}[theorem]{Example}

\newtheorem{defin}[theorem]{Definition}

\theoremstyle{remark}
\newtheorem*{remark}{Remark}

\makeatletter
\newcommand*{\da@rightarrow}{\mathchar"0\hexnumber@\symAMSa 4B }
\newcommand*{\da@leftarrow}{\mathchar"0\hexnumber@\symAMSa 4C }
\newcommand*{\xdashrightarrow}[2][]{%
  \mathrel{%
    \mathpalette{\da@xarrow{#1}{#2}{}\da@rightarrow{\,}{}}{}%
  }%
}
\newcommand{\xdashleftarrow}[2][]{%
  \mathrel{%
    \mathpalette{\da@xarrow{#1}{#2}\da@leftarrow{}{}{\,}}{}%
  }%
}
\newcommand*{\da@xarrow}[7]{%
  \sbox0{$\ifx#7\scriptstyle\scriptscriptstyle\else\scriptstyle\fi#5#1#6\m@th$}%
  \sbox2{$\ifx#7\scriptstyle\scriptscriptstyle\else\scriptstyle\fi#5#2#6\m@th$}%
  \sbox4{$#7\dabar@\m@th$}%
  \dimen@=\wd0 %
  \ifdim\wd2 >\dimen@
    \dimen@=\wd2 %
  \fi
  \count@=2 %
  \def\da@bars{\dabar@\dabar@}%
  \@whiledim\count@\wd4<\dimen@\do{%
    \advance\count@\@ne
    \expandafter\def\expandafter\da@bars\expandafter{%
      \da@bars
      \dabar@ 
    }%
  }%
  \mathrel{#3}%
  \mathrel{%
    \mathop{\da@bars}\limits
    \ifx\\#1\\%
    \else
      _{\copy0}%
    \fi
    \ifx\\#2\\%
    \else
      ^{\copy2}%
    \fi
  }%
  \mathrel{#4}%
}
\makeatother

\title{One-skeleton posets of Bruhat interval polytopes}

\author{Christian Gaetz}
\thanks{The author is supported by a Klarman Postdoctoral Fellowship at Cornell University.}
\address{Department of Mathematics, Cornell University, Ithaca, NY.}
\email{\href{mailto:crgaetz@gmail.com}{{\tt crgaetz@gmail.com}}}

\date{\today}

\begin{document}
\begin{abstract}
Introduced by Kodama and Williams, \emph{Bruhat interval polytopes} are generalized permutohedra closely connected to the study of torus orbit closures and total positivity in Schubert varieties. We show that the 1-skeleton posets of these polytopes are lattices and classify when the polytopes are simple, thereby resolving open problems and conjectures of Fraser, of Lee--Masuda, and of Lee--Masuda--Park. In particular, we classify when generic torus orbit closures in Schubert varieties are smooth.
\end{abstract}
\maketitle

\section{Introduction}

\subsection{Bruhat interval polytopes}
For a permutation $w$ in $S_n$, write $\bs{w}$ for the vector $(w^{-1}(1),\ldots,w^{-1}(n)) \in \R^n$. The \emph{Bruhat interval polytope} $Q_w$ is defined as the convex hull:
\[
Q_w \coloneqq \conv(\{\bs{u} \mid u \preceq w \}) \subset \R^n,
\]
where $\preceq$ denotes Bruhat order on $S_n$ (see Section~\ref{sec:background}). Bruhat interval polytopes were introduced by Kodama and Williams in \cite{kodama-williams}, where it is shown that they are the images under the \emph{moment map} of the \emph{Schubert variety} $X_w$ in the flag variety, and also of the \emph{totally positive part} $X_w^{\geq 0}$ of the Schubert variety. Therefore, the combinatorics of $Q_w$ encodes information about the actions of the torus and positive torus on $X_w$ and $X_w^{\geq 0}$ respectively.

The combinatorics of $Q_w$ was studied further by Tsukerman and Williams \cite{Tsukerman-Williams}, who showed that $Q_w$ is a \emph{generalized permutohedron} in the sense of Postnikov \cite{postnikov-beyond} and the matroid polytope of a flag \emph{positroid}. Additional connections to the geometry of matroids were made in \cite{boretsky-eur-williams}, and Bruhat interval polytopes have also appeared \cite{Williams} in the context of BCFW-bridge decompositions \cite{arkani-hamed_bourjaily_cachazo_goncharov_postnikov_trnka_2016} from physics, and in the study \cite{lee-conjecture, lee-retractions, lee-toric-bip, lee-poincare} of \emph{generic torus orbit closures} $Y_w$ in $X_w$.

\subsection{The 1-skeleton of $Q_w$ as a lattice}
\label{sec:intro-lattice}
Throughout this work, we study the \emph{1-skeleton poset} $P_w$ of $Q_w$, a partial order on the lower Bruhat interval $[e,w]=\{u \mid u \preceq w\}$.

\begin{defin}
\label{def:Pw}
The poset $(P_w, \leq_w)$ has underlying set the Bruhat interval $[e,w]$ and cover relations $u \lessdot_w v$ whenever $Q_w$ has an edge between vertices $\bs{u}$ and $\bs{v}$ and $\ell(v)>\ell(u)$, where $\ell$ denotes Coxeter length. See Figure~\ref{fig:poset-example} for an example.
\end{defin}

\begin{figure}
    \centering
    \includegraphics{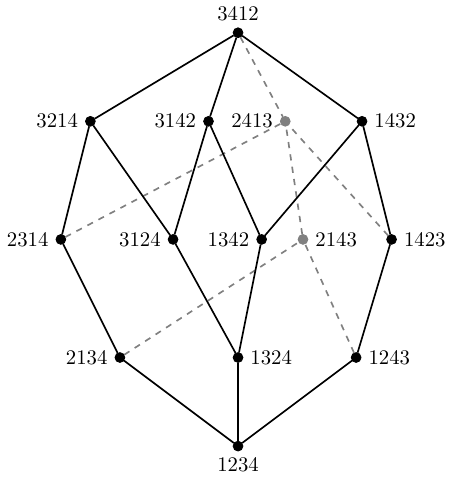}
    \caption{The Hasse diagram of the poset $P_{3412}$. The facial structure of the polytope $Q_{3412}$ may be seen by viewing the black vertices and edges as the ``front" and the gray ones as the ``back".}
    \label{fig:poset-example}
\end{figure}

When $w=w_0$ is the \emph{longest permutation}, the polytope $Q_w$ is the \emph{permutohedron}, a fundamental object in algebraic combinatorics, and the poset $P_w$ is the very well-studied \emph{right weak order} (see Section~\ref{sec:background}). For general $w$, since edges of $Q_w$ must be Bruhat covers by \cite[Thm.~4.1]{Tsukerman-Williams}, the order $\leq_w$ is intermediate in strength between right weak order and Bruhat order on $[e,w]$.

Since the work of Bj\"{o}rner \cite[Thm.~8]{bjorner-lattice} it has been known that the weak order $P_{w_0}$ on $S_n$ is a \emph{lattice}; in our first main theorem, we generalize this to all of the posets $P_w$.

\begin{introthm}[Proven as Theorem~\ref{thm:bip-poset-is-lattice}]
\label{thm:intro-lattice}
Let $w \in S_n$, then $P_w$ is a lattice.
\end{introthm}

As we explain in the remainder of Section~\ref{sec:intro-lattice}, special cases of this lattice structure confirm a conjecture of Fraser \cite[Rmk.~3.7]{fraser2020cyclic}, imply new properties of $Q_w$, and suggest interesting directions for future work.

\subsubsection{BCFW-bridge decompositions}
In the last decade, there has been an explosion of work (see \cite{arkani-hamed_bourjaily_cachazo_goncharov_postnikov_trnka_2016}) relating the physical theory of \emph{scattering amplitudes} to the combinatorics and geometry of the \emph{totally nonnegative Grassmannian} $Gr(k,n)_{\geq 0}$ by way of the \emph{amplituhedron}. In this setting, \emph{on-shell diagrams} from physics correspond to \emph{reduced plabic graphs}, which give parametrizations of an important cell decomposition of $Gr(k,n)_{\geq 0}$ \cite{postnikov2006total}. 

In \cite[\S3.2]{arkani-hamed_bourjaily_cachazo_goncharov_postnikov_trnka_2016} it is shown that reduced plabic graphs for a given cell may be built up recursively using \emph{BCFW-bridge decompositions}. In \cite[Thm.~3.2]{Williams}, Williams showed that these decompositions of plabic graphs correspond to the maximal chains in $P_v$ when $v$ is a Grassmannian permutation, and that this is analogous to the fact that reduced words for the longest permutation $w_0$ correspond to maximal chains in $P_{w_0}$ (weak order). Since weak order is a lattice, Theorem~\ref{thm:intro-lattice} extends this analogy and implies new structure within the set of BCFW-bridge decompositions. That $P_v$ is a lattice for Grassmannian permutations was conjectured by Fraser \cite[Rmk.~3.7]{fraser2020cyclic}. Fraser also conjectured that a larger class of posets, which are not necessarily the 1-skeleton posets of any polytope, are lattices; this problem remains open. 

\subsubsection{Quotients of weak order}
Theorem~\ref{thm:intro-lattice} is proven by realizing $P_w$ as a quotient of weak order $P_{w_0}$ by an equivalence relation $\Theta_w$ which respects the weak order join operation (but does \emph{not} respect the meet operation!) Thus $P_w$ is a \emph{semilattice quotient} of $P_{w_0}$ but not a \emph{lattice quotient}. There are families of very important lattice congruences on weak order \cite{reading-cambrian-lattices, reading-speyer} and lattice quotients and lattice homomorphisms of weak order have been classified \cite{reading-noncrossing, reading-homomorphisms}. This work thus suggests that semilattice quotients and homomorphisms of weak order are an intriguing topic for further study.

\subsubsection{The parabolic map and the mixed meet}
Let $S_n(I)$ denote the Young subgroup of $S_n$ generated by a subset $I$ of the simple reflections. Billey, Fan, and Losonczy proved \cite[Thm.~2.2]{billey-parabolic} that for any $w \in S_n$ the set $S_n(I) \cap [e,w]$ has a unique maximal element $m(w,I)$ under Bruhat order; the map $w \mapsto m(w,I)$ is called the \emph{parabolic map}. Richmond and Slofstra \cite[Thm.~3.3 \& Prop.~4.2]{Richmond-fiber-bundle} showed that this element $m(w,I)$ determines whether the projection of the Schubert variety $X_w \subset G/B$ to a partial flag variety $G/P$ is a fiber bundle, and is thus important for understanding the singularities of $X_w$. We apply Theorem~\ref{thm:intro-lattice} to show in Theorem~\ref{thm:join-of-atoms} that the element $m(w,I)$ is just the join in $P_w$ of the simple reflections from $I$, demonstrating the richness of the lattice structure on $P_w$.

A related operation of \emph{mixed meet} was studied by Bump and Chetard in \cite[Thm.~3]{bump2021matrix} in relation to certain intertwining operators of representations of reductive groups over nonarchimedean local fields. The mixed meet of $u,v \in S_n$ is the unique Bruhat maximal permutation in $[e,u]_R \cap [e,v]$. In the language of Section~\ref{sec:lattice}, this element is $\Bot_v(u)$, the unique minimal element under $\leq_R$ in the equivalence class of $u$ under the equivalence relation $\Theta_v$ induced on $S_n$ by the normal fan of $Q_v$. This element is a translate of $\mu_v(u)$, where $\mu_v$ is the \emph{matroid map} obtained by viewing $[e,v]$ as a Coxeter matroid in the sense of \cite{coxeter-matroids}.

\subsubsection{The non-revisiting path property}
A polytope $Q$ has the \emph{non-revisiting path property} if no shortest path in its 1-skeleton between two vertices returns to a face after having left it. This property has long been of interest in the field of combinatorial optimization. In \cite{Hersh}, Hersh conjectures that any simple polytope whose 1-skeleton poset is a lattice has the non-revisiting path property, and proves several weaker properties satisfied by such polytopes. Thus, combining Theorem~\ref{thm:intro-lattice} and the classification of simple Bruhat interval polytopes in Theorem~\ref{thm:intro-smooth} below, we obtain a rich new family of examples to which Hersh's conjecture and results apply. Additionally, in Section~\ref{sec:directionally-simple} we observe that all polytopes $Q_w$ are \emph{directionally} simple. It is thus natural to ask: can Hersh's conjecture and results be extended to the class of directionally simple polytopes?

\subsection{Bruhat interval polytopes and generic torus orbit closures}

\subsubsection{Simple Bruhat interval polytopes and smooth torus orbit closures}
Let $G=GL_n(\mathbb{C})$, let $B$ denote the Borel subgroup of upper triangular matrices, and let $T$ denote the maximal torus of diagonal matrices. The \emph{flag variety} $Fl_n=G/B$ and its Schubert subvarieties $X_w \coloneqq \overline{BwB/B}$ are of fundamental importance in many areas of algebraic combinatorics, algebraic geometry, and representation theory. The torus $T$ acts naturally on $G/B$ via left multiplication, and the fixed points $(G/B)^T$ are the points $wB$ for $w \in S_n$, where we identify $w$ with its permutation matrix. The fixed points of the Schubert variety $X_w$ are $\{uB \mid u \preceq w\}$.

Torus orbits in $G/B$ and their closures are a rich family of varieties, studied since Klyachko \cite{Klyachko} and Gelfand--Serganova \cite{Gelfand-serganova} with close connections to matroids and Coxeter matroids \cite{coxeter-matroids}. One class of torus orbit closures has received considerable interest \cite{lee-conjecture, lee-retractions, lee-toric-bip, lee-poincare} of late: \emph{generic} torus orbit closures in Schubert varieties. A torus orbit closure $Y \subset X_w$ is called generic if $Y^T=X_w^T$; we write $Y_w$ for any generic torus orbit closure in $X_w$.

One of the main properties of interest for torus orbits in the flag variety has historically been their singularities \cite{Carrell-normality, Carrell-smooth-point} and in particular determining when they are smooth. For Schubert varieties themselves, smoothness was famously characterized by Lakshmibai--Sandhya \cite[Thm.~1]{Lakshmibai-sandhya} in terms of permutation pattern avoidance. In our next main theorem, we resolve a conjecture of Lee and Masuda \cite[Conj.~7.17]{lee-conjecture} by classifying when $Y_w$ is smooth.

\begin{introthm}[Conj.~7.17 of Lee--Masuda \cite{lee-conjecture}; Proven below as Corollary~\ref{cor:smooth-conjecture}]
\label{thm:intro-smooth}
Let $w \in S_n$, then $Q_w$ is a simple polytope if and only if it is simple at the vertex $\bs{w}$; equivalently, $Y_w$ is a smooth variety if and only if it is smooth at the point $wB$.
\end{introthm}

Theorem~\ref{thm:intro-smooth} is proven by showing (see Theorem~\ref{thm:degree-monotonicity}) that the degree of a vertex of $Q_w$ is an ordering preserving function of the poset $P_w$.

By \cite[Cor.~7.13]{lee-conjecture}, the condition that $Y_w$ is smooth at $wB$ can be checked combinatorially by determining whether a certain graph $\Gamma_w(w)$ is a tree (see Section~\ref{sec:gamma-graphs}). This tree condition has in turn been characterized combinatorially in terms of pattern avoidance \cite[Thm.~1.1]{bosquet-melou-forest}, and shown \cite[Prop.~2]{woo-yong-gorenstein} to characterize when $X_w$ is \emph{locally factorial}. By work of Bj\"{o}rner--Ekedahl \cite[Thm.~D]{Bjorner-ekedahl} it is also equivalent to the vanishing of the coefficient of $q$ in the associated \emph{Kazhdan--Lusztig polynomial} \cite{kazhdan--lusztig} and thus \cite{kazhdan--lusztig2} the vanishing of a certain middle intersection cohomology group of $X_w$. It would be fascinating to give a purely geometric explanation for the equivalence (by Theorem~\ref{thm:intro-smooth}) of the smoothness of $Y_w$ with these other geometric conditions on $X_w$. 

While by Theorem~\ref{thm:intro-smooth} the smoothness of $Y_w$ is determined at the ``top" torus fixed point $wB$, the smoothness of $X_w$ is known to be determined at the ``bottom" fixed point $eB$. It would also be interesting to give a geometrically natural explanation for this discrepancy. 

\subsubsection{Directionally simple polytopes and $h$-vectors}

In Section~\ref{sec:directionally-simple} we show that, even when $Q_w$ is not a simple polytope, it is still \emph{directionally simple} (see Definition~\ref{def:directionally-simple}). This fact was also shown in \cite[Prop.~4.5]{lee-poincare} by an involved calculation, but follows directly from our results realizing $P_w$ as a quotient of weak order. This property of $Q_w$ implies that its \emph{h-vector} has positive entries which count certain permutations according to their number of ascents. In Proposition~\ref{prop:smooth-if-palindromic} we resolve an open problem of Lee--Masuda--Park \cite[Prob.~6.1]{lee-survey} by showing that $Y_w$ is smooth if and only if this $h$-vector is palindromic.

\subsubsection{Generalizations to other Bruhat intervals}

Kodama and Williams \cite{kodama-williams} in fact defined Bruhat interval polytopes 
\[
Q_{w',w} \coloneqq \conv(\{\bs{u} \mid w' \preceq u \preceq w \}) \subset \R^n,
\]
for any $w' \preceq w$, generalizing the case $w'=e$ on which we focus in this work. It is natural to ask to what extent the results of this paper can be generalized. The lattice property of Theorem~\ref{thm:intro-lattice} fails for $P_{w',w}$ with $w'=12435$ and $w=35142$, and Theorem~\ref{thm:intro-smooth} also fails to generalize, even in $S_4$. It is possible, though, that both results hold for the class of $Q_{w',w}$ which are simple at $\bs{w'}$ (this includes all $Q_w=Q_{e,w}$). This class of Bruhat interval polytopes is notable for its applications \cite[Prop.~4.4]{barkley2023combinatorial} to the Combinatorial Invariance Conjecture for Kazhdan--Lusztig polynomials.

\subsection{Outline}

Section~\ref{sec:background} contains standard background material on the weak and strong Bruhat orders. In Section~\ref{sec:gamma-graphs} we recall results from \cite{lee-conjecture} relating edges of $Q_w$ to certain directed graphs $\Gamma_w(u)$ and $\tilde{\Gamma}_w(u)$. We establish new combinatorial properties of these graphs, notably Proposition~\ref{prop:v-paths-to-u-paths}, which form the basis for the main results of the paper. In Section~\ref{sec:lattice} we give several new characterizations of the poset $P_w$ and prove Theorem~\ref{thm:intro-lattice} and note several of its consequences.  These results are then applied in Section~\ref{sec:directionally-simple} to reprove the directional simplicity of $Q_w$ and to resolve an open problem posed in \cite{lee-survey}. Finally, in Section~\ref{sec:monotone} we pull together all of our understanding of $\Gamma_w(u)$ and $P_w$ to prove a strengthened version of Theorem~\ref{thm:intro-smooth}.

An extended abstract describing part of this work appears in the proceedings of FPSAC 2023 \cite{fpsac-abstract}.

\section{Background on the weak and strong Bruhat orders}
\label{sec:background}
We refer the reader to \cite{bjorner-brenti} for basic definitions and results on Coxeter groups and the Bruhat and weak orders on them.

We view the symmetric group $S_n$ as a Coxeter group with simple generators $\{s_1,\ldots,s_{n-1}\}$, where $s_i\coloneqq (i \: i+1)$ is an adjacent transposition. An expression $w=s_{i_1}\cdots s_{i_{\ell}}$ of minimal length is a \emph{reduced word} for $w$ and in this case the quantity $\ell=\ell(w)$ is the \emph{length} of $w$. There are three important partial orders on $S_n$, each graded by length. The \emph{right weak order} $\leq_R$ by definition has cover relations $w \lessdot_R ws$ whenever $s$ is a simple generator and $\ell(ws)=\ell(w)+1$; the \emph{left weak order} $\leq_L$ is defined analogously, but with left-multiplication by $s$. The \emph{(strong) Bruhat order} $\preceq$ has cover relations $w \precdot wt$ whenever $\ell(wt)=\ell(w)+1$ and $t$ lies in the set $T$ of transpositions $(ij)$. We write $[v,w]_R$ and $[v,w]$ for the closed interval between $v,w$ in right weak and Bruhat order respectively.

The \emph{left inversions} of an element $w \in S_n$ are the reflections $T_L(w)\coloneqq \{t \in T \: | \: \ell(tw)<\ell(w)\}$ and the \emph{left descents} are $D_L(w)\coloneqq T_L(w) \cap \{s_1,\ldots, s_{n-1}\}$; the \emph{right inversions} and \emph{right descents} are defined analogously, using instead right multiplication by $t$. It is a well-known fact that weak order is characterized by containment of inversion sets:

\begin{prop}[Cor.~3.1.4 of \cite{bjorner-brenti}]
\label{prop:weak-characterized-by-inversions}
Let $v,w \in S_n$, then $v \leq_L w$ if and only if $T_R(v) \subseteq T_R(w)$ and $v \leq_R w$ if and only if $T_L(v) \subseteq T_L(w)$.
\end{prop}

The \emph{(positive) root} associated to the reflection $(ab)$ with $a<b$ is the vector $e_a-e_b$, where the $e_i$ are the standard basis vectors in $\mathbb{R}^n$. We write $\Phi^+$ for the set $\{e_i-e_j \mid 1\leq i < j \leq n\}$ of all positive roots. A set $A \subseteq \Phi^+$ is called \emph{closed} if $\alpha+\beta \in A$ whenever $\alpha,\beta \in A$ and $\alpha+\beta \in \Phi^+$; it is \emph{coclosed} if $\Phi^+ \setminus A$ is closed, and \emph{biclosed} if it is both closed and coclosed. The following well-known fact can be easily verified.

\begin{prop}
A set $T' \subseteq T$ of reflections is the inversion set of some permutation if and only if the associated set $\{e_a-e_b \mid (ab) \in T', a<b\}$ is biclosed.
\end{prop}

The Bruhat order has a useful characterization in terms of reduced words:

\begin{prop}[Thm.~2.2.2 of \cite{bjorner-brenti}]
Let $v,w \in S_n$, then $v \preceq w$ if and only if every reduced word (equivalently, some reduced word) for $w$ has a subword which is a reduced word for $v$.
\end{prop}


Given a string $v_1 \ldots v_k$ where the $v_i$ are distinct numbers from $[n]$, the \emph{flattening} $\fl(v_1\ldots v_k)$ is the permutation $v_1' \cdots v_k' \in S_k$ where $v'_i=m$ if $v_i$ is the $m$-th largest element of $\{v_1,\ldots,v_k\}$. 

\begin{prop}[follows from Thm.~2.6.3 of \cite{bjorner-brenti}]
\label{prop:flattening-bruhat}
Suppose $v,w \in S_n$ satisfy $v_i=w_i$ for $i \not \in A \subseteq [n]$. Let $v_A$ and $w_A$ be the subsequences of $v,w$ consisting of those numbers from $A$, then $v \preceq w$ if and only if $\fl(v_A) \preceq \fl(w_A)$.
\end{prop}

The symmetric group contains a unique element $w_0$ of maximum length, and $w_0$ is the unique maximal element of $S_n$ under each of $\leq_L, \leq_R,$ and $\preceq$. In fact, in the finite case, both left and right weak order are lattices \cite[Thm.~3.2.1]{bjorner-brenti}: each pair $v,w$ of elements has a unique greatest lower bound or \emph{meet} $x \land_L y$ (resp. $x \land_R y$) under $\leq_L$ (resp. $\leq_R$) and a unique least upper bound or \emph{join} $x \lor_L y$ (resp. $x \lor_R y$).


\section{The graphs $\tilde{\Gamma}_w$ and $\Gamma_w$}
\label{sec:gamma-graphs}
\begin{defin}[Def.~7.1 of \cite{lee-conjecture}]
For $u \preceq w$, the directed graph $\tilde{\Gamma}_w(u)$ has vertex set $[n]$ with directed edges $(u(i),u(j))$ whenever $i<j$, $u(ij) \preceq w$, and $\left| \ell(u(ij))-\ell(u) \right|=1$. We write $\tilde{E}_w(u)$ for this set of edges.

The \emph{transitive reduction} of a directed graph $G$ is a directed graph $G'$ on the same vertex set with as few edges as possible, subject to the condition that there is a directed path from $v$ to $w$ in $G$ if and only if there is one in $G'$. The transitive reduction of a finite graph without directed cycles is unique. We define $\Gamma_w(u)$ to be the transitive reduction of $\tilde{\Gamma}_w(u)$, with edge set $E_w(u)$. See Example~\ref{ex:gamma-example}.
\end{defin}

\begin{figure}
    \centering
    \includegraphics{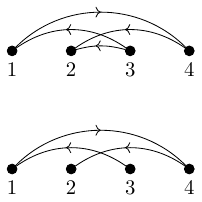}
    \caption{The graphs $\tilde{\Gamma}_{3412}(3142)$ (top) and $\Gamma_{3412}(3142)$ (bottom).}
    \label{fig:gamma-example}
\end{figure}

\begin{prop}[Prop.~7.7 of \cite{lee-conjecture}]
\label{prop:edges-from-E}
Two vertices $\bs{u}$ and $\bs{v}$ of $Q_w$ are connected by an edge of the polytope if and only if $v=u(ij)$ where $(u(i),u(j)) \in E_w(u)$.
\end{prop}

\begin{ex}
\label{ex:gamma-example}
Let $w=3412$ and $u=3142$, then $u(12),u(23),u(34),$ and $u(14)$ each have length $2$ or $4$ and lie below $w$ in Bruhat order. These correspond to the edges $(3,1),(1,4),(4,2),$ and $(3,2)$ in $\tilde{\Gamma}_w(u)$ (see Figure~\ref{fig:gamma-example}). Of these, $(3,2)$ is not an edge of $\Gamma_w(u)$, since the other three edges give another directed path from $3$ to $2$. The remaining edges $(3,1),(1,4),$ and $(4,2)$ of $\Gamma_w(u)$ imply by Proposition~\ref{prop:edges-from-E} that $\bs{3142}$ is connected by an edge of $Q_w$ to $\bs{1342},\bs{3412},$ and $\bs{3124}$, in agreement with Figure~\ref{fig:poset-example}.
\end{ex}

\subsection{Basic properties}
\label{sec:arrows-properties}
When the permutation $w$ is understood, we write $a \xrightarrow{u} b$ when $(a,b) \in \tilde{E}_w(u)$ and $a \xdashrightarrow{u} b$ when there is a directed path from $a$ to $b$ in $\tilde{\Gamma}_w(u)$ (equivalently, in $\Gamma_w(u)$); we write $a \xRightarrow{u} b$ when $(a,b) \in E_w(u)$.

\begin{prop}
\label{prop:gamma-graph-basic-properties}
Let $u,w \in S_n$ with $u \preceq w$, then:
\begin{itemize}
    \item[\normalfont (i)] $\Gamma_w(u)$ and $\tilde{\Gamma}_w(u)$ have no directed cycles, and
    \item[\normalfont (ii)] $\Gamma_w(u)$ contains no triangles (of any orientation).
\end{itemize}
\end{prop}
\begin{proof}
For any edge $a \xrightarrow{u} b$, we have by definition that $u^{-1}(a)<u^{-1}(b)$, so $\tilde{\Gamma}_w(u)$ and $\Gamma_w(u)$ cannot contain directed cycles. Any other orientation of a triangle is not transitively reduced, so $\Gamma_w(u)$ does not contain any triangles. 
\end{proof}

\begin{prop}
\label{prop:reflection-gives-gamma-path}
Let $u,w \in S_n$ with $u \preceq w$, and suppose $(ab)u \preceq w$, where $u^{-1}(a)<u^{-1}(b)$, then $a \xdashrightarrow{u} b$. In particular, if $(ab) \in T_L(u)$, then $a \xdashrightarrow{u} b$.
\end{prop}
\begin{proof}
If there is no $i$ with $u^{-1}(a)<i<u^{-1}(b)$ and $\min(a,b)<u(i)<\max(a,b)$, then $|\ell((ab)u)-\ell(u)|=1$, so $a \xrightarrow{u} b$. Otherwise, find the smallest such $i$, for which we clearly have $a \xrightarrow{u} i$. By induction on $|\ell((ab)u)-\ell(u)|$, we have that $i \xdashrightarrow{u} b$, so $a \xdashrightarrow{u} b$.
\end{proof}

\subsection{Local changes}
Throughout this section we suppose that $u,v \in S_n$ satisfy $u \lessdot_w v = (ab)u$, with $a<b$, which by Proposition~\ref{prop:edges-from-E} implies that $(a,b) \in E_w(u)$ and $(b,a) \in E_w(v)$. Our goal is to understand how the graphs $\tilde{\Gamma}_w(u)$ and $\tilde{\Gamma}_w(v)$ differ. The following proposition will be fundamental in the remainder of the paper.

\begin{prop}
\label{prop:v-paths-to-u-paths}
Suppose that $u,v \in S_n$ satisfy $u \lessdot_w v = (ab)u$, then:
\begin{itemize}
    \item[\normalfont(i)] If $c<d$ and $c \xdashrightarrow{v} d$, then $c \xdashrightarrow{u} d$;
    \item[\normalfont(ii)] If $c>d, c \neq b, d \neq a,$ and $c \xdashrightarrow{v} d$, then $c \xdashrightarrow{u} d$.
\end{itemize}
\end{prop}

We will prove Proposition~\ref{prop:v-paths-to-u-paths} after a series of lemmas.

\begin{lemma}
\label{lem:notes-lemma-4}
Suppose that $u,v \in S_n$ satisfy $u \lessdot_w v = (ab)u$, then:
\begin{itemize}
    \item[\normalfont{(i)}] If $c \xrightarrow{v} b$ then $c \xdashrightarrow{u} a$;
    \item[\normalfont{(ii)}] If $a \xrightarrow{v} d$ then $b \xdashrightarrow{u} d$.
\end{itemize}
\end{lemma}
\begin{proof}
Since $c \xrightarrow{v} b \xrightarrow{v} a$, we have $v=\ldots c \ldots b \ldots a \ldots$ and since $u=(ab)v$, we have $u=\ldots c \ldots a \ldots b \ldots$. The fact that $c \xrightarrow{v} b$ implies by definition that $(bc)v \preceq w$. Now, $(bc)v=\ldots b \ldots c \ldots a \ldots$ while $(ac)u=\ldots a \ldots c \ldots b \ldots$. By Proposition~\ref{prop:flattening-bruhat} and since $a<b$ we have $(ac)u \prec (bc)v \preceq w$. Finally, by Proposition~\ref{prop:reflection-gives-gamma-path} we get $c \xdashrightarrow{u} a$. The proof of (ii) is exactly analogous.
\end{proof}

\begin{lemma}
\label{lem:c-to-d-through-a}
Suppose that $u,v \in S_n$ satisfy $u \lessdot_w v = (ab)u$ and that $c \xrightarrow{v} a \xrightarrow{v} d$, then $c \xdashrightarrow{u} d$.
\end{lemma}
\begin{proof}
By Lemma~\ref{lem:notes-lemma-4}(ii) we have $b \xdashrightarrow{u} d$, so if $c=b$ we are done. Thus assume $c \neq b$; since $b \xrightarrow{v} a$, this leaves two possibilities for $v$, omitting the ellipses: $v=cbad$ or $v=bcad$.

Consider first the case $v=cbad$, in which case $u=cabd$. If $c>a$, then $c \xdashrightarrow{u} a$ by Proposition~\ref{prop:reflection-gives-gamma-path}, so $c \xdashrightarrow{u} a \xrightarrow{u} b \xdashrightarrow{u} d$. Thus assume $c<a$. In this case, we have $(ac)u=acbd \prec w$ by Proposition~\ref{prop:flattening-bruhat} since $c<a<b$ and $(ac)v=abcd \preceq w$. Thus by Proposition~\ref{prop:reflection-gives-gamma-path} we have $c \xdashrightarrow{u} a \xrightarrow{u} b \xdashrightarrow{u} d$.

Consider now the case $v=bcad$ and $u=acbd$. First suppose $c<b$, in this case we have $b \xdashrightarrow{v} c \xrightarrow{v} a$, contradicting the fact that $b \xRightarrow{v} a$. Thus $c>b$ so $c \xdashrightarrow{u} b \xdashrightarrow{u} d$.
\end{proof}

\begin{lemma}
\label{lem:c-to-d-through-b}
Suppose that $u,v \in S_n$ satisfy $u \lessdot_w v = (ab)u$ and that $c \xrightarrow{v} b \xrightarrow{v} d$, then $c \xdashrightarrow{u} d$.
\end{lemma}
\begin{proof}
By Lemma~\ref{lem:notes-lemma-4}(i) we have $c \xdashrightarrow{u} a$, so if $d=a$ we are done. Otherwise, there are two possibilities for $v$, omitting ellipses: $v=cbad$ or $v=cbda$.

If $v=cbad$ then $u=cabd$. Consider $(bd)u=cadb$: we have $(bd)u \prec cbda$, and, if $d<a$ we have $cbda \prec v \preceq w$, so $b \xdashrightarrow{u} d$ by Proposition~\ref{prop:reflection-gives-gamma-path}; if instead $d>a$ then $(bd)u \prec cdab = (bd) v \prec w$, so again $b \xdashrightarrow{u} d$. Thus in either case we have $c \xdashrightarrow{u} a \xrightarrow{u}b \xdashrightarrow{u} d$.

If $v=cbda$ and $u=cadb$, then $(ad)u=cdab \prec cdba = (bd)v \preceq w$. Thus $c \xdashrightarrow{u} a \xdashrightarrow{u} d$.
\end{proof}

\begin{lemma}
\label{lem:new-path-goes-through-a-b}
Suppose that $u,v \in S_n$ satisfy $u \lessdot_w v = (ab)u$ and that
\[
c \xrightarrow{v} i_1 \xrightarrow{v} \cdots \xrightarrow{v} i_k \xrightarrow{v} d
\]
with $c,i_1,\ldots,i_k,d \not \in \{a,b\}$, then $c \xdashrightarrow{u}d$.
\end{lemma}
\begin{proof}
It suffices to prove the case $k=0$, since the relation $\xdashrightarrow{u}$ is transitive, so suppose $c \xrightarrow{v} d$ with $c,d \not \in \{a,b\}$. By definition, we know $v^{-1}(c)<v^{-1}(d), (cd)v \preceq w,$ and $\left| \ell((cd)v) - \ell(v) \right|=1$. Since $c,d \not \in \{a,b\}$, we know $u^{-1}(c)<u^{-1}(d)$ and $(cd)u \preceq (cd)v \preceq w$. If in addition we have $\left| \ell((cd)u) - \ell(u) \right|=1$, then $c \xrightarrow{u} d$ and we are done. Otherwise, it must be that the values $a,b,c,d$ appear in $u$ in the relative order $acbd$, with $\min(c,d)<b<\max(c,d)$ or in the relative order $cadb$ with $\min(c,d)<a<\max(c,d)$. In the first case we have $c \xrightarrow{u} b \xrightarrow{u}d$ and in the second case we have $c \xrightarrow{u} a \xrightarrow{u} d$.
\end{proof}

We are now ready to give the proof of Proposition~\ref{prop:v-paths-to-u-paths}.

\begin{proof}[Proof of Proposition~\ref{prop:v-paths-to-u-paths}]
Suppose that $u,v \in S_n$ satisfy $u \lessdot_w v = (ab)u$, with $a<b$, and suppose $c \xdashrightarrow{v} d$. Suppose first that $c,d \not \in \{a,b\}$ and consider a path in $\tilde{\Gamma}_w(v)$ from $c$ to $d$. If the path does not pass through $a$ nor $b$, then $c \xdashrightarrow{u} d$ by Lemma~\ref{lem:new-path-goes-through-a-b}. If the path passes through $a$, so $c \xdashrightarrow{v} c' \xrightarrow{v} a \xrightarrow{v} d' \xdashrightarrow{v} d$, then applying Lemma~\ref{lem:new-path-goes-through-a-b} and Lemma~\ref{lem:c-to-d-through-a} we get $c \xdashrightarrow{u} c' \xdashrightarrow{u} d' \xdashrightarrow{u} d$. Similarly, if the path passes through $b$, or through both $b$ and $a$ using the edge $b \xrightarrow{v} a$, we can conclude $c \xdashrightarrow{u} d$ by applying Lemmas~\ref{lem:notes-lemma-4}-\ref{lem:new-path-goes-through-a-b}.

It only remains to consider the cases where $\{c, d\} \cap \{a,b\} \neq \emptyset$. We will cover the cases $c=a$ or $b$, with the situation for $d=a$ or $b$ being symmetrical (note that if $(c,d)=(b,a)$ then neither part of Proposition~\ref{prop:v-paths-to-u-paths} applies, and indeed we have $d \xrightarrow{u} c$ instead).

If $c=a \xrightarrow{v} d_1 \xrightarrow{v} \cdots \xrightarrow{v} d_k=d$, then we have $c=a \xrightarrow{u} b \xdashrightarrow{u} d_1 \xdashrightarrow{u} d_k=d$ by Lemmas~\ref{lem:notes-lemma-4}(ii) and \ref{lem:new-path-goes-through-a-b}. If $c=b \xrightarrow{v} a \xrightarrow{v} d_1 \xrightarrow{v} \cdots \xrightarrow{v} d_k=d$, we have $c=b\xdashrightarrow{u} d_1 \xdashrightarrow{u} d_k=d$ by Lemmas~\ref{lem:notes-lemma-4} and \ref{lem:new-path-goes-through-a-b}. If $c=b \xrightarrow{v} d_1 \xrightarrow{v} \cdots \xrightarrow{v} d_k=d$ and $b>d$, then neither case of the proposition applies.

Thus the last case to consider is $c=b \xrightarrow{v} d_1 \xrightarrow{v} \cdots \xrightarrow{v} d_k=d$ with $b<d$ and $a \neq d_1$. If $d_1>a$ and $v^{-1}(a)<v^{-1}(d_1)$, then $(b d_1)u \prec (b d_1)v \preceq w$, so $b \xdashrightarrow{u} d_1 \xdashrightarrow{u} d_k$ by Proposition~\ref{prop:reflection-gives-gamma-path} and Lemma~\ref{lem:new-path-goes-through-a-b}. If $d_1>a$ and $v^{-1}(a)>v^{-1}(d_1)$ then $b \xrightarrow{v} d_1 \xdashrightarrow{v} a$, contradicting the assumption that $b \xRightarrow{v} a$. If $d_1<a$ and $v^{-1}(a)<v^{-1}(d_1)$, then we have $a \xdashrightarrow{v} d_1$, so we may instead consider a path $b \xrightarrow{v} a \xdashrightarrow{v} d_1 \xrightarrow{v} \cdots \xrightarrow{v} d_k$ and apply a previous case. Finally, if $d_1<a$ and $v^{-1}(d_1)<v^{-1}(a)$, consider the smallest $i$ such that $v^{-1}(d_i)>v^{-1}(a)$ (this exists, since $a<b<d$, so we must have $v^{-1}(d)>v^{-1}(a)$ to avoid contradicting $b \xRightarrow{v} a$). Then $d_{i-1}<a$ (otherwise $b \xdashrightarrow{v} d_{i-1} \xdashrightarrow{v} a$ would contradict $b \xRightarrow{v} a$) and so we must have $d_i<a$, since otherwise $(d_{i-1}d_i)$ would not give a Bruhat cover of $v$. Thus $a \xdashrightarrow{v} d_i$, so $c=b \xdashrightarrow{u} d_i \xdashrightarrow{u} d$ by Lemmas~\ref{lem:notes-lemma-4} and Lemma~\ref{lem:new-path-goes-through-a-b}. 
\end{proof}

\section{The lattice property}
\label{sec:lattice}

\subsection{Generalized permutohedra}

The \emph{normal fan} $N(Q)$ of a polytope $Q \subset \R^n$ (see e.g. \cite[Ex.~7.3]{ziegler-lectures}) is the fan in $\R^n$ with a cone $C(F)$ for each nonempty face $F$ of $Q$ with
\[
C(F) = \{ x \in \R^n \mid F \subseteq \argmax_{x' \in Q} \langle x, x' \rangle \}.
\]
The correspondence $F \mapsto C(F)$ is an order-reversing bijection from the poset of faces of $Q$ under containment to the poset of cones of $N(Q)$ under containment.

The normal fan of the permutohedron $\Perm_n = Q_{w_0}$ is the fan determined by the \emph{braid arrangement}, which has defining hyperplanes $x_i-x_j=0$ for $1 \leq i < j \leq n$. The top-dimensional cones $C_{w_0}(\bs{y})$ in this fan are naturally labelled by permutations $y \in S_n$ which give the relative order of the coordinates of a point $(x_1,\ldots,x_n) \in C_{w_0}(\bs{y})$. In particular we have $\bs{y} \in C_{w_0}(\bs{y})$. 

Following Postnikov \cite{postnikov-beyond}, a polytope $Q$ such that cones of $N(Q)$ are unions of cones of $N(\Perm_n)$ is called a \emph{generalized permutohedron}. Kodama--Williams \cite[Cor.~A.8]{kodama-williams} showed that Bruhat interval polytopes are generalized permutohedra. Let $C_w(\bs{u})$ denote the top-dimensional cone of $N(Q_w)$ corresponding to the vertex $\bs{u} \in Q_w$ (where $u \in [e,w]$). Each $C_w(\bs{u})$ is a union of some of the $C_{w_0}(\bs{y})$; viewing these as equivalence classes on the $y$, we obtain an equivalence relation $\Theta_w$ on $S_n$. We write $[y]_w$ for the equivalence class of $y$ under $\Theta_w$. 

We say $y \in S_n$ is a \emph{linear extension} of $\Gamma_w(u)$ (equivalently, of $\tilde{\Gamma}_w(u)$) if $y^{-1}(i)<y^{-1}(j)$ whenever $i \xdashrightarrow{u} j$. The following proposition is immediate from the construction of $\Gamma_w(u)$ in \cite[\S7]{lee-conjecture} and the discussion of normal fans of generalized permutohedra in \cite[\S3]{postnikov-reiner-williams}.

\begin{prop}
\label{prop:classes-are-linear-extensions-of-gamma}
Let $w \in S_n$ and $u \preceq w$, then $[u]_w$ is exactly the set of linear extensions of $\Gamma_w(u)$.
\end{prop}

Somewhat surprisingly, the equivalence classes $[x]_w$ turn out to be intervals in right weak order. This result was established by other means in \cite[Prop.~4.3]{lee-conjecture}.

\begin{prop}
\label{prop:class-has-min-max-elements}
Let $x,w \in S_n$, then there exist elements $\Bot_w(x)$ and $\Top_w(x)$ such that $[x]_w=[\Bot_w(x),\Top_w(x)]_R$.
\end{prop}
\begin{proof}
Let $u$ be the unique element of $[e,w] \cap [x]_w$. By Proposition~\ref{prop:classes-are-linear-extensions-of-gamma}, the elements $y$ of $[x]_w$ are exactly the linear extensions of $\tilde{\Gamma}_w(u)$. Suppose that $(ab) \in T_L(u)$ with $a<b$, then by Proposition~\ref{prop:reflection-gives-gamma-path} we have $b \xdashrightarrow{u} a$, so by Proposition~\ref{prop:classes-are-linear-extensions-of-gamma} we have $(ab) \in T_L(y)$ for any $y \in [x]_w$. Thus by Proposition~\ref{prop:weak-characterized-by-inversions} we have $u \leq_R y$, so $\Bot_w(x)=u$.

The reflections occurring as left inversions of some linear extension of $\tilde{\Gamma}_w(u)$ are exactly those in 
\[
I \coloneqq \{(ab) \mid a<b \text{ and } a \not \xdashrightarrow{u} b\}.
\]
To see that a unique maximum $\Top_w(x)$ exists, we will demonstrate that $R=\{e_a-e_b \mid (ab) \in I\}$ is biclosed, so that $\Top_w(x)$ will be the unique permutation with left inversion set $I$.

First, note that if $a \xdashrightarrow{u} b$ and $b \xdashrightarrow{u} c$, then $a \xdashrightarrow{u} c$, so $R$ is coclosed. 

For closedness, let $a<b<c$ and assume that $a \xdashrightarrow{u} c$, which implies that $u^{-1}(a)<u^{-1}(c)$. If $u^{-1}(b)<u^{-1}(a)$, then $(ab) \in T_L(u)$, so by Proposition~\ref{prop:reflection-gives-gamma-path} we have $b \xdashrightarrow{u} a \xdashrightarrow{u} c$, so $b \xdashrightarrow{u} c$. If instead $u^{-1}(b)>u^{-1}(c)$, then $(bc) \in T_L(u)$, so by Proposition~\ref{prop:reflection-gives-gamma-path} we have $a \xdashrightarrow{u} c \xdashrightarrow{u} b$, so $a \xdashrightarrow{u} b$. Otherwise we have $u^{-1}(a)<u^{-1}(b)<u^{-1}(c)$. Consider a path $a \to a_1 \to \cdots \to a_r \to c_1 \to \cdots \to c_s \to c$, where $u^{-1}(a_i)\leq u^{-1}(b)$ and $u^{-1}(b)<u^{-1}(c_j)$ for all $i,j$. If any $a_i>b$, then $a \xdashrightarrow{u} a_i \xdashrightarrow{u} b$. If any $c_j<b$, then $b \xdashrightarrow{u} c_j \xdashrightarrow{u} c$. Otherwise, since $(a_r \: c_1) u$ covers $u$ in Bruhat order, we must have $a_r=b$, so $a \xdashrightarrow{u} b \xdashrightarrow{u} c$. In all cases, we see $a \xdashrightarrow{u} b$ or $b \xdashrightarrow{u} c$, so $R$ is closed.  
\end{proof}

\subsection{The poset structure}

Write $\weak_R(S_n)$ for right weak order on $S_n$. The quotient $\weak_R(S_n)/\Theta_w$ is the relation $\leq_{\Theta_w}$ defined by setting $[x]_w \leq_{\Theta_w} [y]_w$ whenever there exist $x' \in [x]_w$ and $y' \in [y]_w$ with $x' \leq_R y'$.

\begin{theorem}
\label{thm:top-is-order-preserving}
Given $w \in S_n$, the map $\Top_w: S_n \to S_n$ is order preserving with respect to right weak order. That is, if $x \leq_R y$ then $\Top_w(x) \leq_R \Top_w(y)$. Furthermore, $\weak_R(S_n)/\Theta_w$ is isomorphic to $P_w$ via the map $[x]_w \mapsto \Bot_w(x)$.
\end{theorem}
\begin{proof}
Suppose that $x \lessdot_R y=xs=tx$ this implies that $C_{w_0}(\bs{x})$ and $C_{w_0}(\bs{y})$ share a facet along the hyperplane $H_t$ fixed by the reflection $t$. Suppose further that $[x]_w \neq [y]_w$ and let $u = \Bot_w(x)$ and $v= \Bot_w(y)$. Thus $C_w(\bs{u})$ and $C_w(\bs{u})$ share a facet along $H_t$, so there is an edge of $Q_w$ with vertices $\bs{u}$ and $\bs{v}$. This implies that $v=t'u$ for some $t' \in T$ and that $C_w(\bs{u})$ and $C_w(\bs{v})$ share a facet along $H_{t'}$.  Since the convex cones $C_w(\bs{u})$ and $C_w(\bs{v})$ share at most one facet, we must have in fact that $t=t'$. We have $u \leq_R x$ by Proposition~\ref{prop:class-has-min-max-elements} and we know that $t \not \in T_L(u)$ by Proposition~\ref{prop:weak-characterized-by-inversions} and the fact that $t \not \in T_L(x)$. Thus $\ell(v)>\ell(u)$ and $u \lessdot_w v$. 

Now, by the proof of Proposition~\ref{prop:class-has-min-max-elements}, we have for any $z \in S_n$ that 
\begin{equation}
\label{eq:inversions-of-top-from-graph}
T_L(\Top_w(z))=\{(cd) \mid c<d \text{ and } c \not \xdashrightarrow{\Bot_w(z)} d\}.
\end{equation}
Thus for any $u' \lessdot_w v'=(ab)u$ with $a<b$ we have by Proposition~\ref{prop:v-paths-to-u-paths}(i) and (\ref{eq:inversions-of-top-from-graph}) that 
\[
T_L(\Top_w(u')) \subset T_L(\Top_w(v')).
\]
By Proposition~\ref{prop:weak-characterized-by-inversions} we see $\Top_w(u') <_R \Top_w(v')$. 

This establishes that $P_w \cong \weak_R(S_n)/\Theta_w$, and establishes that $\Top_w$ is order preserving after applying the second paragraph and that fact that $\Top_w(u)=\Top_w(x)$ and $\Top_w(v)=\Top_w(y)$.
\end{proof}

\begin{cor}
\label{cor:bip-poset-is-induced-weak}
Let $w \in S_n$, then the map $\Top_w$ is a poset isomorphism from $P_w$ to $(\Top_w([e,w]), \leq_R)$.
\end{cor}
\begin{proof}
It is clear by definition that $\Top_w$ is injective on $[e,w]$, since each equivalence class $[x]_w$ contains a unique element $\Bot_w(x)$ of $[e,w]$, thus it is a bijection onto its image. For $u,v \in [e,w]$ with $u \leq_w v$, by Theorem~\ref{thm:top-is-order-preserving} there exist $u' \in [u]_w$ and $v' \in [v]_w$ with $u' \leq_R v'$. Then Theorem~\ref{thm:top-is-order-preserving} gives that 
\[
\Top_w(u)=\Top_w(u') \leq_R \Top_w(v')=\Top_w(v).
\]

Note that for $v \in [e,w]$ we have $\Bot_w(\Top_w(v))=v$ and that Theorem~\ref{thm:top-is-order-preserving} implies that $\Bot_w$ sends right weak order relations to order relations under $\leq_w$. Thus if $\Top_w(u) \leq_R \Top_w(v)$ then $u \leq_w v$.
\end{proof}

The fact that $P_w$ is a lattice also follows easily from Theorem~\ref{thm:top-is-order-preserving}.

\begin{theorem}
\label{thm:bip-poset-is-lattice}
For any $w \in S_n$, the poset $P_w$ is a lattice, with join operation given by
\[
u \lor_w v = \Bot_w(\Top_w(u) \lor_R \Top_w(v)).
\]
\end{theorem}
\begin{proof}
Let $z=\Bot_w(\Top_w(u) \lor_R \Top_w(v))$. Then 
\[
u \leq_R \Top_w(u) \leq_R \Top_w(u) \lor_R \Top_w(v),
\]
so by Theorem~\ref{thm:top-is-order-preserving} we have $u \leq_w z$, and similarly $v \leq_w z$. On the other hand, if $y \geq_w u,v$, then by Theorem~\ref{thm:top-is-order-preserving} we have $\Top_w(y) \geq \Top_w(u), \Top_w(v)$ so $\Top_w(y) \geq \Top_w(u) \lor_R \Top_w(v)$. Thus $y \geq_w z$ and we see that $z$ is the join of $u,v$ in $P_w$. Since $P_w$ is a finite poset with a join and a unique minimal element (namely $e$), it also has a meet and is thus a lattice.
\end{proof}

\subsection{The Billey--Fan--Losonczy parabolic map}

Let $S_n(I)$ denote the Young subgroup of $S_n$ generated by a subset $I$ of the simple reflections. For $w \in S_n$ let $m(w,I)$ denote the unique maximal element of $S_n(I) \cap [e,w]$ under Bruhat order \cite[Thm.~2.2]{billey-parabolic}; $w \mapsto m(w,I)$ is called the \emph{parabolic map}. See Richmond and Slofstra \cite[Thm.~3.3 \& Prop.~4.2]{Richmond-fiber-bundle} for the importance of the parabolic map in determining the fiber bundle structure of Schubert varieties.

\begin{prop}
\label{prop:join-of-atoms-is-top}
Let $w \in S_n$, and let $s_{i_1},\ldots,s_{i_k}$ be the simple reflections appearing in some (equivalently, any) reduced word for $w$, then:
\[
s_{i_1} \lor_w \cdots \lor_w s_{i_k} = w.
\]
\end{prop}
\begin{proof}
By Theorem~\ref{thm:bip-poset-is-lattice} we have
\[
s_{i_1} \lor_w \cdots \lor_w s_{i_k} = \Bot_w(\Top_w(s_{i_1}) \lor_R \cdots \lor_R \Top_w(s_{i_k})). 
\]
Now, $\Top_w(s_{i_1}) \lor_R \cdots \lor_R \Top_w(s_{i_k}) \geq_R s_{i_1} \lor_R \cdots \lor_R s_{i_k}= w_0(J)$ where $J=\{s_{i_1},\ldots, s_{i_k}\}$ and where $w_0(J)$ denotes the unique longest element of the subgroup of $S_n$ generated by $J$ (here we have used \cite[Lem.~3.2.3]{bjorner-brenti} for the equality). Thus $\Bot_w(\Top_w(s_{i_1}) \lor_R \cdots \lor_R \Top_w(s_{i_k})) \geq_w \Bot_w(w_0(J))=w$, where this last equality follows since $w_0(J) \geq_R u$ for all $u \in [e,w]$. Since $w$ is the maximal element of $P_w$, we get the desired equality.
\end{proof}

\begin{theorem}
\label{thm:join-of-atoms}
Let $w \in S_n$, and let $I$ be a set of simple generators, then:
\[
m(w,I)=\bigvee_w \{s_i \in I \mid s_i \preceq w\}.
\]
\end{theorem}
\begin{proof}
Let $I'=\{s_i \in I \mid s_i \preceq w\}$; clearly any $s \in I$ with $s \not \preceq w$ affects neither $m(w,I)$ nor the join, so we may reduce to the case $I'=I$. If $I=\{s_1,\ldots,\widehat{s_a},\ldots,s_{n-1}\}$ is a maximal proper subset of the simple reflections, then the set of vertices $\bs{u}$ of $Q_w$ for $u \in S_n(I) \cap [e,w]$ can be cut out by the hyperplanes
\[
\{(x_1,\ldots,x_n) \in \mathbb{R}^n \mid \sum_{i=1}^a x_i = \sum_{i=1}^a i\}
\]
and
\[
\{(x_1,\ldots,x_n) \in \mathbb{R}^n \mid \sum_{i=a+1}^n x_i = \sum_{i=a+1}^n i\}.
\]
Since these are supporting hyperplanes of $Q_w$, this set of vertices are the vertices of some face $F_I$ of $Q_w$. If $I$ is not maximal, we can obtain a face by intersecting faces for maximal subsets.

Since faces of $Q_w$ containing $e$ are themselves of the form $Q_y$ by \cite[Thm.~4.1]{Tsukerman-Williams}, we see that $m(w,I)=y$ exists. Now, $P_y$ is an interval of, and thus a sublattice of, $P_w$. Thus, since 
\[
\bigvee_y I = y = m(w,I)
\]
by Proposition~\ref{prop:join-of-atoms-is-top}, we obtain the desired result.
\end{proof}

\section{Directionally simple polytopes}
\label{sec:directionally-simple}
Given a polytope $Q \subset \R^d$, say that a cost vector $c \in \R^d$ is \emph{generic} if $c$ is not orthogonal to any edge of $Q$. A generic cost vector induces an acyclic orientation on the 1-skeleton $G(Q)$ by taking edges to be oriented in the direction of greater inner product with $c$; we write $G_c(Q)$ for the resulting acyclic directed graph. It is clear that every face $F$ of $Q$ contains a unique source $\min_c(F)$ and sink $\max_c(F)$ with respect to this orientation.

\begin{defin}
\label{def:directionally-simple}
We say that a polytope $Q \subset \R^d$ is \emph{directionally simple} with respect to the generic cost vector $c$ if for every vertex $v$ of $Q$ and every set $E$ of edges of $G_c(Q)$ with source $v$ there exists a face $F$ of $Q$ containing $v$ whose set of edges incident to $v$ is exactly $E$.
\end{defin}

Since any subset of the edges incident to a vertex $v$ in a simple polytope spans a face, the following fact is clear:

\begin{prop}
A simple polytope $Q \subset \R^d$ is directionally simple with respect to any generic cost vector.
\end{prop}

\subsection{$Q_w$ is directionally simple}

Theorem~\ref{thm:bip-is-directionally-simple} was proven in \cite[Prop.~4.5]{lee-poincare} by an involved direct computation; here we give a new proof using the results of Section~\ref{sec:lattice}.

\begin{theorem}
\label{thm:bip-is-directionally-simple}
Let $w \in S_n$, then $Q_w$ is a directionally simple polytope with respect to the cost vector $c=(n,n-1,\ldots,1)$.
\end{theorem}
\begin{proof}
The vector $c$ is chosen so that $G_c(Q_w)$ coincides with the Hasse diagram of $P_w$, with the outward edges from a vertex $\bs{u}$ corresponding to the upper covers of $u$ in $P_w$. By Theorem~\ref{thm:top-is-order-preserving} the set $A=\{z \in S_n \mid \Top_w(u) \lessdot_R z\}$ of weak order upper covers of $\Top_w(u)$ is in bijection with the set $B=\{v \in P_w \mid u\lessdot_w v\}$ of upper covers of $u$ in $P_w$, via the map $z \mapsto \Bot_w(z)$. Since weak order is the 1-skeleton poset of the permutohedron $\Perm_n$, and since $\Perm_n$ is a simple polytope, for any $E \subseteq A$, there is a face $F$ of $\Perm_n$ whose collection of edges incident to $\bs{\Top_w(u)}$ are exactly those connecting $\bs{\Top_w(u)}$ to $\bs{z}$ for $z \in E$. Since $Q_w$ is a generalized permutohedron, the set $\{\bs{\Bot_w(z)} \mid \bs{z} \in F\}$ are the vertices of some face $F'$ of $Q_w$, witnessing the upper simplicity of $Q_w$.
\end{proof}

Theorem~\ref{thm:bip-is-directionally-simple} shows that $Q_w$ is always directionally simple, in Section~\ref{sec:monotone} we will determine when $Q_w$ is in fact simple.

\subsection{$h$-vectors of directionally simple polytopes}
The \emph{f-vector} of a polytope $Q \subset \R^d$ is the tuple $f(Q)=(f_0,\ldots,f_d)$ where $f_i$ is the number of $i$-dimensional faces of $Q$. The \emph{h-vector} $h(Q)$ is defined by the equality of polynomials
\begin{equation}
\label{eq:f-and-h-vectors}
\sum_{i=0}^d f_i(x-1)^i = \sum_{k=0}^d h_k x^k.
\end{equation}

\begin{prop}
\label{prop:dir-simple-h-vec}
Let $Q \subset \R^d$ be directionally simple with respect to the generic cost vector $c$, with $h$-vector $h(Q)=(h_0,\ldots,h_d)$. Then for all $k=0,\ldots,d$ the entry $h_k$ is the number of vertices of $Q$ with out-degree exactly $k$ in $G_c(Q)$.
\end{prop}
\begin{proof}
First note that, since the cost vector $c$ is generic, each face of $Q$ has a unique minimal vertex, that is, a vertex which is a source when we restrict the graph $G_c(Q)$ to the vertices from $Q$. Let $g_k$ be the number of vertices of $Q$ with out-degree exactly $k$ in $G_c(Q)$. We can count $i$-faces according to their bottom vertex. Each vertex with out-degree $k$, by upper simplicity, is the bottom vertex of exactly ${k \choose i}$ faces. Thus we have:
\[
\sum_{i=0}^d f_i x^i = \sum_{k=0}^d g_k (x+1)^k.
\]
But this is just a reparametrization of (\ref{eq:f-and-h-vectors}), so $g_k=h_k$.
\end{proof}

\begin{remark}
One implication of Proposition~\ref{prop:dir-simple-h-vec} is that $h_i \geq 0$ for $i=0,\ldots,d$. This by itself is already a very special property of directionally simple polytopes; indeed $h$-vectors of non-simple polytopes are rarely considered, because they are rarely positive or otherwise interesting.
\end{remark}

For $u \in S_n$, write $\asc(u)$ for the number $n-1-|D_R(u)|$ of right ascents of $u$. Corollary~\ref{cor:h-vec-of-bip} below is an extension to $Q_w$ of the kind of interpretation for $h$-vectors of \emph{simple} generalized permutohedra given by Postnikov--Reiner--Williams \cite[Thm.~4.2]{postnikov-reiner-williams}. 

\begin{cor}
\label{cor:h-vec-of-bip}
Let $(h_0,h_1,\ldots)$ be the $h$-vector of $Q_w$, then for all $k$ we have:
\[
h_k = \left| \{ z \in \Top_w([e,w]) \mid \asc(z)=k\} \right|.
\]
\end{cor}
\begin{proof}
As explained in the proof of Theorem~\ref{thm:bip-is-directionally-simple}, the number of ascents of $\Top_w(u)$ is exactly the out-degree on $\bs{u}$ in $G_c(Q_w)$, so the result follows by Proposition~\ref{prop:dir-simple-h-vec}.
\end{proof}

As explained in \cite{lee-poincare}, the $h$-vector of $Q_w$ also gives the Poincar\'{e} polynomial of the toric variety $Y_w$, so Corollary~\ref{cor:h-vec-of-bip} gives a new formula for that invariant. We can also resolve an open problem raised in \cite{lee-survey}:

\begin{prop}[Resolves Problem 6.1 of \cite{lee-survey}]
\label{prop:smooth-if-palindromic}
The variety $Y_w$ is smooth if and only if its Poincar\'{e} polynomial is palindromic.
\end{prop}
\begin{proof}
Suppose $Q_w$ is $d$-dimensional. Since $Q_w$ is directionally simple (see Theorem~\ref{thm:bip-is-directionally-simple}), by \cite[Thm.~2.7]{lee-poincare} the Poincar\'{e} polynomial of $Y_w$ has coefficients $h_0,h_1,\ldots, h_d$. Suppose that this sequence is palindromic.

By (\ref{eq:f-and-h-vectors}) the number of vertices of $Q_w$ is $f_0=\sum_{k=0}^d h_k$ and the number of edges of $Q_w$ is $f_1=\sum_{k=0}^d k \cdot h_k$. Since $h$ is assumed to be palindromic, we in fact have:
\begin{align}
    f_0 &= \begin{cases} 2 \sum_{k=0}^{\frac{d-1}{2}} h_k, \text{ $d$ odd} \\ 2 \sum_{k=0}^{\frac{d}{2}-1} h_k + h_{\frac{d}{2}}, \text{ $d$ even}. \end{cases} \\
    f_1 &= \begin{cases} d \sum_{k=0}^{\frac{d-1}{2}} h_k, \text{ $d$ odd} \\ d \sum_{k=0}^{\frac{d}{2}-1} h_k + \frac{d}{2} h_{\frac{d}{2}}, \text{ $d$ even}. \end{cases}
\end{align}
In particular, we have $f_1=\frac{d}{2}f_0$. Since all vertices are incident to at least $d$ edges in a $d$-dimensional polytope, this implies that in fact all vertices are incident to exactly $d$ edges, so $Q_w$ is in fact simple, which by \cite[Thm.~1.2]{lee-conjecture} implies that $Y_w$ is smooth. The converse is immediate, since if $Y_w$ is smooth its cohomology satisfies Poincar\'{e} duality.
\end{proof}

\section{Vertex-degree monotonicity}
\label{sec:monotone}

In Section~\ref{sec:lattice} we applied properties of the relation $c \xdashrightarrow{u} d$ to prove that $P_w$ is a lattice. In this section we use more refined information about the relation $c \xRightarrow{u} d$ (see Section~\ref{sec:arrows-properties}) to prove that vertex-degrees of $Q_w$ are monotonic with respect to the partial order $\leq_w$; as an application, we resolve a conjecture of Lee--Masuda \cite[Conj.~7.17]{lee-conjecture} characterizing smooth generic torus orbit closures in Schubert varieties.

Write $\deg_w(\bs{u})$ for the number of edges of $Q_w$ incident to the vertex $\bs{u}$.

\begin{theorem}
\label{thm:degree-monotonicity}
Let $w \in S_n$. If $u \leq_w v$ then $\deg_w(\bs{u}) \leq \deg_w(\bs{v})$.
\end{theorem}

Theorem~\ref{thm:degree-monotonicity} will follow from the stronger Theorem~\ref{thm:injection} below.

\begin{cor}
\label{cor:simple-conjecture}
Let $w \in S_n$, then the polytope $Q_w$ is simple if and only if it is simple at the vertex $\bs{w}$. 
\end{cor}
\begin{proof}
It is clear from Proposition~\ref{prop:edges-from-E} and the definition of $E_w(e)$ that $Q_w$ is always simple at the vertex $\bs{e}$. Thus if $Q_w$ is also simple at $\bs{w}$, Theorem~\ref{thm:degree-monotonicity} imples that it is simple at every vertex.
\end{proof}

Corollary~\ref{cor:simple-conjecture} resolves Conjecture 7.17 of Lee--Masuda \cite{lee-conjecture}. As described in \cite[Cor.~7.13]{lee-conjecture}, Corollary~\ref{cor:simple-conjecture} has the following geometric interpretation.

\begin{cor}
\label{cor:smooth-conjecture}
Let $Y_w$ be a generic torus orbit closure in the Schubert variety $X_w\coloneqq \overline{BwB/B}$, then $Y_w$ is smooth if and only if it is smooth at the torus fixed point $wB$.
\end{cor}

Write $c \xLeftrightarrow{u} d$ if $c \xLeftarrow{u} d$ or $d \xRightarrow{u} c$ (note that we never have both $c \xLeftarrow{u} d$ and $d \xRightarrow{u} c$).

\begin{theorem}
\label{thm:injection}
Let $w \in S_n$ and suppose $u \lessdot_w v = tu$ with $c \xLeftrightarrow{u} d$, then there is a unique edge $\Vec{e}$ of $E_w(v)$ described by $c \xLeftrightarrow{v} d$ or $t(c) \xLeftrightarrow{v} t(d)$. Moreover, the map
\[
\ph:E_w(u) \to E_w(v)
\]
sending the edge $c \xLeftrightarrow{u} d$ to $\Vec{e}$ is an injection.
\end{theorem}

Theorem~\ref{thm:degree-monotonicity} follows from Theorem~\ref{thm:injection} since, by Proposition~\ref{prop:edges-from-E} we have $\deg_w(\bs{u})=\left| E_w(u) \right|$ for all $u \preceq w$.

\subsection{Proof of Theorem~\ref{thm:injection}}

\subsubsection{Injectivity}
\label{sec:proof-injective}

\begin{lemma}
\label{lem:at-most-one-edge}
In the setting of Theorem~\ref{thm:injection}, at most one edge of $E_w(v)$ is described by $c \xLeftrightarrow{v} d$ or $t(c) \xLeftrightarrow{v} t(d)$.
\end{lemma}
\begin{proof}
Write $t=(ab)$ with $a<b$. If $\left|\{a,b\} \cap \{c,d\} \right| \neq 1$, then $\{c,d\}=\{t(c),t(d)\}$ so the conditions $c \xLeftrightarrow{v} d$ and $t(c) \xLeftrightarrow{v} t(d)$ are the same, and clearly we can have $c \xRightarrow{v} d$ or $d \xRightarrow{v} c$ but not both. If $\left| \{a,b\} \cap \{c,d\} \right|=1$, suppose for example that $b=c$, so $t(c)=a$ and $t(d)=d$. Suppose we had both edges, then since $u \lessdot_w v=tu$ we would have
\[
b \xRightarrow{v} a=t(c) \xLeftrightarrow{v} t(d)=d \xLeftrightarrow{v} c=b,
\]
contradicting the fact that $\Gamma_w(v)$ has no triangles by Proposition~\ref{prop:gamma-graph-basic-properties}. The other cases are analogous.
\end{proof}

In Sections~\ref{sec:proof-downward} and \ref{sec:proof-upward} below we show that an edge in $E_w(v)$ described by $c \xLeftrightarrow{v} d$ or $t(c) \xLeftrightarrow{v} t(d)$ does in fact exist, so that $\ph:E_w(u) \to E_w(v)$ is well-defined. It remains to check that $\ph$ is injective.

\begin{lemma}
\label{lem:injective}
In the setting of Theorem~\ref{thm:injection}, the map $\ph:E_w(u) \to E_w(v)$ is injective.
\end{lemma}
\begin{proof}
Write $t=(ab)$ with $a<b$. Suppose $\ph((c,d))=\ph((c',d'))=(i,j) \in E_w(v)$ with $(c,d) \neq (c',d')$. Without loss of generality we have
\[
\{c,d\}=\{i,j\}=\{t(c'),t(d')\}. 
\]
Since $\{t(c'),t(d')\}=\{c,d\} \neq \{c',d'\}$ we must have $\left| \{c',d'\} \cap \{a,b\} \right|=1$. Assume for example that $c'=a$ (the other cases being analogous) which implies that $t(c')=b$ and $t(d')=d'$. Then we have 
\[
a \xRightarrow{u} b=t(c') \xLeftrightarrow{u} t(d') = d' \xLeftrightarrow{u} c'=a,
\]
contradicting the fact that $\Gamma_w(u)$ has no triangles by Proposition~\ref{prop:gamma-graph-basic-properties}.
\end{proof}

\subsubsection{Upward edges}
\label{sec:proof-upward}
When the edge in Theorem~\ref{thm:injection} is an upward edge $c \xRightarrow{u} d$ with $c<d$, the result will follow from the following classification of 2-dimensional faces in Bruhat interval polytopes, due to Williams \cite{Williams}.

\begin{theorem}[Thm.~5.1 of \cite{Williams}]
\label{thm:2-face-classification}
A 2-dimensional face of a Bruhat interval polytope is either a square, trapezoid, or regular hexagon, with labels as in Figure~\ref{fig:2-faces}.
\end{theorem}

\begin{figure}
    \centering
    \begin{tikzpicture}[scale=0.8]
    \node at (0,0) {$\bullet$};
    \node at (-1,1) {$\bullet$};
    \node at (1,1) {$\bullet$};
    \node at (0,2) {$\bullet$};
    \draw (0,0) -- (-1,1) node [near start, label=left:$(ij)$] {};
    \draw (0,0) -- (1,1) node [near start, label=right:$(kl)$] {};
    \draw (0,2) -- (-1,1) node [near start, label=left:$(kl)$] {};
    \draw (0,2) -- (1,1) node [near start, label=right:$(ij)$] {};
    \node at (3,0) {$\bullet$};
    \node at (3,1.41) {$\bullet$};
    \node at (4.22,2.12) {$\bullet$};
    \node at (5.44,1.41) {$\bullet$};
    \draw (3,0)--(3,1.41) node [midway, label=left:$(jk)$] {};
    \draw (3,1.41)--(4.22,2.12) node [near start, label=above:$(ij)$] {};
    \draw (4.22,2.12)--(5.44,1.41) node [near start, label=right:$(jk)$] {};
    \draw (5.44,1.41)--(3,0) node [near end, label=right:$(ik)$] {};
    \node at (-3,2.12) {$\bullet$};
    \node at (-3,.71) {$\bullet$};
    \node at (-4.22,0) {$\bullet$};
    \node at (-5.44,.71) {$\bullet$};
    \draw (-3,2.12)--(-3,.71) node [midway, label=right:$(ij)$] {};
    \draw (-3,.71)--(-4.22,0) node [near start, label=below:$(jk)$] {};
    \draw (-4.22,0)--(-5.44,.71) node [near start, label=left:$(ij)$] {};
    \draw (-5.44,.71)--(-3,2.12) node [near end, label=left:$(ik)$] {};
    \hspace{0.2in}
    \node at (8.5,0) {$\bullet$};
    \node at (7,.7) {$\bullet$};
    \node at (7,2.11) {$\bullet$};
    \node at (8.5,2.82) {$\bullet$};
    \node at (10,.7) {$\bullet$};
    \node at (10,2.11) {$\bullet$};
    \draw (8.5,0)--(7,.7) node [near start, label=left:$(ij)$] {};
    \draw (7,.7)--(7,2.11) node [midway, label=left:$(jk)$] {};
    \draw (7,2.11)--(8.5,2.82) node [near end, label=left:$(ij)$] {};
    \draw (8.5,2.82)--(10,2.11) node [near start, label=right:$(jk)$] {};
    \draw (10,2.11)--(10,.7) node [midway, label=right:$(ij)$] {};
    \draw (10,.7)--(8.5,0) node [near end, label=right:$(jk)$] {};
    \end{tikzpicture}
    \caption{The possible 2-dimensional faces of a Bruhat interval polytope, according to Theorem~\ref{thm:2-face-classification}; the reflections of these about a vertical line are also possible. The edge labels indicate the reflection which sends one endpoint to the other. For the square, we must have $i<j<k<l$; for the trapezoids, we must have $i<j<k$ or $k<j<i$; for the hexagon, we must have $i<j<k$.}
    \label{fig:2-faces}
\end{figure}
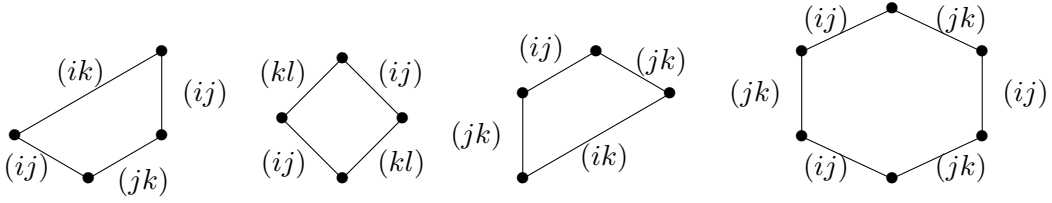

Indeed, in this case the two upward edges incident to $\bs{u}$, coming from $u \lessdot_w v$ and $c \xRightarrow{u} d$ span a 2-dimensional face of $Q_w$, by Theorem~\ref{thm:bip-is-directionally-simple}. Then, viewing $\bs{u}$ as the bottom vertex of one of the faces in Figure~\ref{fig:2-faces} and $\bs{v}$ as one of the vertices covering it, it is easy to verify that in each case Theorem~\ref{thm:injection} holds.

\subsubsection{Downward edges}
\label{sec:proof-downward}
In the previous section we established the well-definedness of the map in Theorem~\ref{thm:injection} when the edge $c \xLeftrightarrow{u} d$ was an upward edge, meaning it points from a smaller index to a larger. In this section, we cover the downward edges; this requires a more careful analysis because it is no longer the case that there is some 2-dimensional face of $Q_w$ containing the edges under consideration.

Throughout this section, let $w \in S_n$ and suppose that $u \lessdot_w v = (ab)u$, with $a<b$. 


\begin{lemma}
\label{lem:square-downward-edge}
Suppose that $d \xRightarrow{u} c$ where $c<d$ and $a,b,c,d$ are distinct. Then $d \xRightarrow{v} c$.
\end{lemma}
\begin{proof}
We have $\ell((cd)v)=\ell(v)-1$, since no value between $c$ and $d$ occurs between them in $v$, for otherwise the same would be true in $u$. Thus $d \xrightarrow{v} c$, so if we are not to have $d \xRightarrow{v} c$, then there must be an index $i$ with $d \xdashrightarrow{v} i \xrightarrow{v} c$. We must have $i \in \{a,b\}$, otherwise by Proposition~\ref{prop:v-paths-to-u-paths} we would have $d \xdashrightarrow{u} i \xdashrightarrow{u} c$, contradicting $d \xRightarrow{u} c$. Consider two cases: $i=a$ or $i=b$.

Suppose $i=a$, so $i=a \xdashrightarrow{u} c$ by Proposition~\ref{prop:v-paths-to-u-paths}. If $d<a$ then we have $d \xdashrightarrow{u} a \xdashrightarrow{u} c$ by Proposition~\ref{prop:v-paths-to-u-paths}, a contradiction, so assume $d>a$. There are two possibilities for $v$ (omitting ellipses): $v=bdac$ or $v=dbac$. If $v=dbac$ then $u=dabc$ so again $d \xdashrightarrow{u} a \xdashrightarrow{u} c$, so assume $v=bdac$ and $u=adbc$. Since multiplication by $(ab)$ and $(cd)$ both give lower Bruhat covers of $v$, and since $(cd)$ gives a lower cover of $u$, we must have $a<b<c<d$. But then we have $d \xdashrightarrow{u} b \xdashrightarrow{u} c$, both by Proposition~\ref{prop:reflection-gives-gamma-path}. This again contradicts $d \xRightarrow{u} c$.

Suppose now that $i=b$, so $d \xdashrightarrow{u} b=i$ by Proposition~\ref{prop:v-paths-to-u-paths}. If $b<c$ then we are done since $d \xdashrightarrow{u} b \xdashrightarrow{u} c$ again by Proposition~\ref{prop:v-paths-to-u-paths}, so assume $b>c$. There are two possibilities for $v$ (omitting ellipses): $v=dbac$ or $v=dbca$. If $v=dbac$ then $u=dabc$ so $b \xdashrightarrow{u} c$, again a contradiction. Thus we must have $v=dbca$ and $u=dacb$, and the known Bruhat covers of $u,v$ imply that $c<d<a<b$. But then by Proposition~\ref{prop:v-paths-to-u-paths} we have $d \xdashrightarrow{u} a$ (since $d \xdashrightarrow{v} b \xrightarrow{v} a$) and we have $a \xdashrightarrow{u} c$ by Proposition~\ref{prop:reflection-gives-gamma-path}, again contradicting $d \xRightarrow{u} c$.
\end{proof}

\begin{lemma}
\label{lem:downward-hex-face}
Suppose that $c \xRightarrow{u} a$, with $a<c$, then $c \xRightarrow{v} a$ or $c \xRightarrow{v} b$.
\end{lemma}
\begin{proof}

We have $c \xRightarrow{u} a \xRightarrow{u} b$, so (omitting ellipses) $u=cab$ and $v=cba$.

Suppose first that $c \xdashrightarrow{v} b$; if we are not to have $c \xRightarrow{v} b$, then it must be that $c \xdashrightarrow{v} i \xdashrightarrow{v} b$ for some $i$. Thus $v=ciba$ and $u=ciab$. By Proposition~\ref{prop:v-paths-to-u-paths} we have $c\xdashrightarrow{u} i$. If $i>a$, then by Proposition~\ref{prop:reflection-gives-gamma-path} we have $i\xdashrightarrow{u} a$, contradicting $c \xRightarrow{u} a$, thus $i<a$. But $i \xdashrightarrow{v} b \xrightarrow{v} a$, so by Proposition~\ref{prop:v-paths-to-u-paths} we again obtain $i \xdashrightarrow{u} a$.

Now assume $c \not \xdashrightarrow{v} b$; in particular, this means that $c<b$. We have $c \xdashrightarrow{v} a$ by Proposition~\ref{prop:reflection-gives-gamma-path}, so if we are not to have $c \xRightarrow{v} a$, it must be that $c \xrightarrow{v} i \xdashrightarrow{v} a$ for some $i \neq b$. By Proposition~\ref{prop:v-paths-to-u-paths} we have $c \xdashrightarrow{u} i$. If $i<a$, then $i \xdashrightarrow{u} a$, contradicting $c \xRightarrow{u} a$. Thus $i>a$. If $v=cbia$ and $u=caib$, we must in fact have $i>b$, since $(ab)$ is a Bruhat cover. Since $c \xrightarrow{v} i$, we have that $(ci)v=ibca \preceq w$. Thus $(ai)u=ciab \preceq w$, since $ciab \prec ibca$. But by Proposition~\ref{prop:reflection-gives-gamma-path} this implies that $a \xdashrightarrow{u} i \xdashrightarrow{u} b$, contradicting $a \xRightarrow{u} b$. Finally, if instead we have $v=ciba$ and $u=ciab$, then $c \xdashrightarrow{u} i \xdashrightarrow{u} a$ by Proposition~\ref{prop:reflection-gives-gamma-path}, contradicting $c \xRightarrow{u} a$. 
\end{proof}

\begin{lemma}
\label{lem:downward-trap-edge-1}
Suppose $c \xRightarrow{u} b$, with $a<b<c$, then $c \xRightarrow{v} a$.
\end{lemma}
\begin{proof}
There are two possibilities for $u$ (omitting ellipses): $u=acb$ or $u=cab$. The latter cannot occur, since by Proposition~\ref{prop:reflection-gives-gamma-path} we would have $c \xdashrightarrow{u} a \xRightarrow{u} b$, contradicting $c \xRightarrow{u} b$, thus $u=acb$ and $v=bca$. 

We have $c \xdashrightarrow{v} a$ by Proposition~\ref{prop:reflection-gives-gamma-path}, so if we are not to have $c \xRightarrow{v} a$, it must be that $c \xdashrightarrow{v} i \xdashrightarrow{v} a$ for some $i$. Thus $v=bcia$ and $u=acib$. By Proposition~\ref{prop:v-paths-to-u-paths}, we have $c \xdashrightarrow{u} i$; we must have $i<b$, otherwise we would have $i \xdashrightarrow{u} b$, contradicting $c \xRightarrow{u} b$. We cannot have $a<i<b$, since $(ab)$ is a Bruhat cover, so $i<a$. But now Proposition~\ref{prop:v-paths-to-u-paths} implies that $i \xdashrightarrow{u} a$, impossible since $u^{-1}(a)<u^{-1}(i)$.
\end{proof}


\begin{proof}[Proof of Theorem~\ref{thm:injection}]
Let $w \in S_n$ and suppose $u \lessdot_w v = (ab)u$ with $a<d$ and that $c \xRightarrow{u} d$. We first argue that at least one edge described by $c \xLeftrightarrow{v} d$ or $t(c) \xLeftrightarrow{v} t(d)$ exists in $E_w(v)$.

The case of an upward edge ($c<d$) was covered in Section~\ref{sec:proof-upward}, so suppose that $c>d$. There are five cases to consider:
\begin{itemize}
    \item[(i)] $a,b,c,d$ are distinct,
    \item[(ii)] $d=a$,
    \item[(iii)] $d=b$,
    \item[(iv)] $c=a$,
    \item[(v)] $c=b$.
\end{itemize}

Conjugation by $w_0$ is an automorphism of Bruhat order (see \cite[Prop.~2.3.4]{bjorner-brenti}), it follows from the definitions that $\tilde{\Gamma}_{w_0ww_0}(w_0uw_0)$ and $\Gamma_{w_0ww_0}(w_0uw_0)$ can be obtained from $\tilde{\Gamma}_{w}(u)$ and $\Gamma_{w}(u)$ respectively by relabelling the vertices according to $i \mapsto n+1-i$ and then reversing all edge directions. Cases (ii) and (iii) correspond under this symmetry to (v) and (iv), respectively, so we only need consider (i),(ii), and (iii). These cases are covered by Lemmas~\ref{lem:square-downward-edge}, \ref{lem:downward-hex-face}, and \ref{lem:downward-trap-edge-1} respectively.

We have shown that at least one edge described by $c \xLeftrightarrow{v} d$ or $t(c) \xLeftrightarrow{v} t(d)$ exists in $E_w(v)$. Lemma~\ref{lem:at-most-one-edge} implies that at most one exists. Together this means that the map $\ph:E_w(u) \to E_w(v)$ is well-defined, and it is injective by Lemma~\ref{lem:injective}.
\end{proof}

\section*{Acknowledgements}
I wish to thank Lauren Williams, Alexander Postnikov, Nathan Reading, Patricia Hersh, Chris Fraser, Grant Barkley, Allen Knutson, Vic Reiner, and Richard Stanley for helpful conversations. I am also grateful to the anonymous referees for their excellent suggestions, and in particular for the corrections and expository improvements to the proof of Theorem~\ref{thm:injection} that they provided.

\bibliographystyle{plain}
\bibliography{bip2}
\end{document}